\documentclass{article}

\usepackage[usenames,dvipsnames]{pstricks}
\usepackage{epsfig}
\usepackage{pst-grad} 
\usepackage{pst-plot} 
\usepackage[space]{grffile} 
\usepackage{etoolbox} 
\makeatletter 
\patchcmd\Gread@eps{\@inputcheck#1 }{\@inputcheck"#1"\relax}{}{}
\makeatother

\usepackage[overload]{empheq}

 \usepackage{mathtools}

\usepackage{amsmath}
\usepackage{amssymb}
\usepackage{amsthm}
\usepackage[latin1]{inputenc}
\usepackage{graphicx}

\usepackage[T1]{fontenc}

\usepackage{ifthen}
\usepackage{color}

\newcommand{\intav}[1]{\mathchoice {\mathop{\vrule width 6pt height 3 pt depth  -2.5pt
\kern -8pt \intop}\nolimits_{\kern -6pt#1}} {\mathop{\vrule width
5pt height 3  pt depth -2.6pt \kern -6pt \intop}\nolimits_{#1}}
{\mathop{\vrule width 5pt height 3 pt depth -2.6pt \kern -6pt
\intop}\nolimits_{#1}} {\mathop{\vrule width 5pt height 3 pt depth
-2.6pt \kern -6pt \intop}\nolimits_{#1}}}

\def\polhk#1{\setbox0=\hbox{#1}{\ooalign{\hidewidth\lower1.5ex\hbox{`}\hidewidth\crcr\unhbox0}}}

\def\XXint#1#2#3{{\setbox0=\hbox{$#1{#2#3}{\int}$ }
\vcenter{\hbox{$#2#3$ }}\kern-.6\wd0}}

\newcommand{\tr}{\operatorname{Tr}}

\usepackage{setspace}

\newtheorem{Theorem}{Theorem}
\newtheorem{example}{Example}
\newtheorem{Definition}{Definition}
\newtheorem{Lemma}{Lemma}

\newtheorem{Proposition}{Proposition}
\newtheorem{Remark}{Remark}
\newtheorem{Assumption}{A}

\makeatletter
\patchcmd{\env@cases}{1.2}{.8}{}{}
\makeatother

\begin{document}

\title{Existence of solutions to a fully nonlinear free transmission problem}
\author{Edgard A. Pimentel and Andrzej \'Swi\k{e}ch}

\date{\today} 

\maketitle

\begin{abstract}

\noindent We study an equation governed by a discontinuous fully nonlinear operator. Such discontinuities are solution-dependent, which introduces a free boundary. Working under natural assumptions, we prove the existence of $L^p$-viscosity and strong solutions to the problem. The operator does not satisfy the usual structure conditions and to obtain the existence of solutions we resort to solving approximate problems, combined with a fixed-point argument. We believe our strategy can be applied to other classes of non-variational free boundary problems.

\medskip


\medskip 

\end{abstract}

\vspace{.1in}

\section{Introduction}\label{sec_introduction}

We consider the problem 

\begin{equation}\label{eq_main}
	\begin{cases}
		F_1(D^2u)\chi_{\{u>0\}}+F_2(D^2u)\chi_{\{u<0\}}=f(x)&\hspace{.1in}\mbox{in}\hspace{.1in}\left(\Omega^+(u)\cup\Omega^-(u)\right)\cap \Omega\\
		u=g&\hspace{.1in}\mbox{on}\hspace{.1in}\partial\Omega,
	\end{cases}
\end{equation}
where $F_i:S(d)\to\mathbb{R}, i=1,2$ are $(\lambda,\Lambda)$-elliptic operators, $f\in L^p(\Omega)$ for some $p>p_0$, and $g\in C(\partial\Omega)$. Here, 
$\Omega^+(u):=\{x\in \Omega\mid u(x)>0\}$, and  $\Omega^-(u):=\{x\in \Omega\mid u(x)<0\}$, $S(d)\approx\mathbb{R}^\frac{d(d+1)}{2}$ is the space of $d\times d$ symmetric matrices and $\frac{d}{2}\leq p_0=p_0(\frac{\Lambda}{\lambda},d)<d$ is the exponent such that the Aleksandrov-Bakelman-Pucci type maximum principle holds for $(\lambda,\Lambda)$-elliptic equations with right hand side in $L^p$ for $p>p_0$, see
\cite{Fabes-Stroock1984,Escauriaza1993,Cabre1995,Fok1998,Crandall-Swiech2003}. 

The main result of the manuscript is the proof of the existence of $L^p$-viscosity solutions to \eqref{eq_main}, Theorem \ref{thm_existence1}. We also show that under an additional, rather natural, condition on the operators $F_1$ and $F_2$, the $L^p$-viscosity solutions belong to $W^{2,p}_{\rm{loc}}(\Omega)$ and are strong solutions to \eqref{eq_main}.

The problem in \eqref{eq_main} accounts for a discontinuous operator, whose discontinuities depend on the sign of the solutions. Within $\{u>0\}$, the equation is governed by $F_1$, whereas in $\{u<0\}$ it is driven by $F_2$. Due to this feature of the model, we localize it in the context of free transmission problems; see \cite{Amaral-Teixeira2015,HPRS2020}. The nomenclature follows from the fact that the interface -- across which discontinuities occur -- can be regarded as a free boundary. It is convenient to notice that such ingredient does not appear in the usual formulation of transmission problems; see \cite{Picone1954, Stampacchia1956, Borsuk2010}, to mention just a few. 

In that case, the transmission interface is known a priori. An important question arising in that setting concerns the regularity of solutions and the dependence of the associated estimates on the geometry of subdomains \cite{Li-Vogelius2000,Li-Nirenberg2003}. Indeed, it has become apparent that the geometry of the interface and the regularity of the solutions are close-knit properties, as recently indicated in \cite{CSCS2021}.

Interior regularity of $L^d$-strong solutions to \eqref{eq_main}, as well as a preliminary analysis of the free boundary, have been studied in \cite{Pimentel-Santos2020}. Nonetheless, the issue of existence is not discussed in that paper and it remained open so far.

A consequential aspect of \eqref{eq_main} concerns the dependence of the operator on the solutions. We observe the equation does not satisfy the usual structure conditions (e.g., \cite[Condition (SC)]{CCKS1996}) under which existence of $L^p$-viscosity solutions is well known \cite{CKLS1999, JS, S1997}. As a result, the existence of $L^p$-viscosity solutions to \eqref{eq_main} requires a different strategy, which is based on the approach employed in \cite{HPRS2020}. 

We start by considering regularized auxiliary problems. In contrast to \eqref{eq_main}, these problems are uniformly elliptic in the entire domain and do not depend on the solutions. Under a usual condition on the geometry of $\Omega$, and the uniform ellipticity assumption on the operators $F_1$ and $F_2$, we produce a family of solutions to such regularized auxiliary problem. The family of such problems depends on a small parameter $\varepsilon$; however, important information (e.g. estimates and moduli of continuity) are found to be uniform for the family. We construct $L^p$-viscosity solutions to the regularized problems by the Schauder Fixed Point Theorem and then  send $\epsilon\to 0$ to produce an $L^p$-viscosity solution to \eqref{eq_main}. If we also suppose $F_1$ and $F_2$ to be close, in a suitable sense, we derive $C^{1,\alpha}$ interior estimates for such solutions. Our first main result reads as follows.

\begin{Theorem}[Existence of viscosity solutions]\label{thm_existence1}
Let Assumptions \emph{(A\ref{assump_domain})}, \emph{(A\ref{assump_ellipticity})} hold, $g\in C(\partial\Omega)$ and $f\in L^p(\Omega)$, for some $p>p_0$.
 Then there exists an $L^p$-viscosity solution $u\in C(\overline\Omega)$ to \eqref{eq_main}. Suppose further that \emph{(A\ref{assump_comp})} also holds and $p>d$; let $\alpha\in(0,1)$ satisfy
 \[
 	\alpha<\alpha_0\hspace{.2in}\mbox{and}\hspace{.2in}\alpha\leq 1-\frac{d}{p},
 \]
 where $\alpha_0\in(0,1)$ corresponds to the $C^{1,\alpha_0}$-regularity available for the solutions to $G=0$ for any $(\lambda,\Lambda)$-elliptic operator $G$. Then there exists $\beta_0=\beta_0(d,p,\lambda,\Lambda,\alpha)>0$ such that if the parameter $\tau>0$ in \emph{(A\ref{assump_comp})} satisfies $\tau\leq \beta_0$, then $u\in C^{1,\alpha}_{\rm{loc}}(\Omega)$ and, for every 
 $\Omega'\Subset\Omega$, we have
 \begin{equation}\label{eq_c1aest}
\|u\|_{C^{1,\alpha}(\Omega')}\leq C\left(1+|F_1(0)|+|F_2(0)|+\|f\|_{L^p(\Omega)}
+\|g\|_{L^\infty(\partial\Omega)}\right),
\end{equation}
where $C=C(\alpha,d,p,\lambda,\Lambda,K,\tau,{\rm diam}(\Omega),{\rm dist}(\Omega',\partial\Omega))$.
\end{Theorem}

Once the existence of $L^p$-viscosity solutions has been established, a natural question concerns the existence of strong solutions to \eqref{eq_main}. Here, our strategy is to examine conditions on the operators $F_1$ and $F_2$ leading to $W^{2,p}$ interior regularity estimates. The almost-convexity of $F_1$ and $F_2$ allows to prove such estimates. To be precise, we suppose that $F_1$ and $F_2$ are locally close to a convex, $(\lambda,\Lambda)$-elliptic operator $F$. 

This additional condition is transmitted through the structure of our proofs; as a consequence, the $L^p$-viscosity solution found in Theorem \ref{thm_existence1} is in $W^{2,p}_{\rm{loc}}(\Omega)$, with the usual estimates. The closeness regime imposed on $F_1$ and $F_2$ is encoded by constants $L$ and $\sigma$; see Assumption (A\ref{assump_convexity}). The existence of strong solutions to \eqref{eq_main} is the content of our second main result.

\begin{Theorem}[Existence of strong solutions]\label{thm_existence2}
Let Assumptions \emph{(A\ref{assump_domain})}, \emph{(A\ref{assump_ellipticity})}, \emph{(A\ref{assump_convexity})} hold, $g\in C(\partial\Omega)$ and $f\in L^p(\Omega)$, for some $p>p_0$. There exists $\beta_0=\beta_0(d,p,\lambda,\Lambda)>0$ such that if the parameter $\sigma>0$ in \emph{(A\ref{assump_convexity})} satisfies $\sigma\leq \beta_0$, then equation \eqref{eq_main} has a strong solution $u\in W^{2,p}_{\rm{loc}}(\Omega)\cap C(\overline\Omega)$. In addition, for every open subset $\Omega'\Subset \Omega$
\[
\|u\|_{W^{2,p}(\Omega')}\leq C\left(1+|F_1(0)|+|F_2(0)|+\|f\|_{L^p(\Omega)}
+\|g\|_{L^\infty(\partial\Omega)}\right),
\]
where $C=C(d,p,\lambda,\Lambda,L,\sigma,{\rm diam}(\Omega),{\rm dist}(\Omega',\partial\Omega))$.
\end{Theorem}

Our strategy bypasses the dependence of the operator on the solutions through a fixed-point argument. We believe this approach may be useful for a larger class of free boundary problems of non-variational nature.

The remainder of this paper is organized as follows: in Section \ref{sec_prelim} we detail our main assumptions and gather preliminary notions and facts used in the paper. Section \ref{sec_proofthm1} presents the proof of Theorem \ref{thm_existence1}. Finally, in Section \ref{sec_proofthm2} we put forward the proof of Theorem \ref{thm_existence2}.

\section{Preliminaries}\label{sec_prelim}

In the sequel we detail the main assumptions under which we work and present preliminary notions used in our arguments. We start  by recalling a definition.

\begin{Definition}[Uniform exterior cone condition]\label{def_cone}
Let $\Omega\subset\mathbb{R}^d$ be open and bounded. We say $\Omega$ satisfies a uniform exterior cone condition if there exist $r,\theta>0$ such that, for every $x\in\partial\Omega$, one can find a cone $C$ of opening $\theta$ and center at the origin, satisfying
\[
	(x+C)\cap B_r(x)\subset\mathbb{R}^d\setminus\Omega.
\]
\end{Definition}

Our arguments require the domain $\Omega$ to satisfy a uniform exterior cone condition. This is stated in the following assumption.

\begin{Assumption}[Regularity of $\Omega$]\label{assump_domain}
The set $\Omega\subset\mathbb{R}^d$ is a bounded domain satisfying a uniform exterior cone condition. 
\end{Assumption}

In addition to the geometry of the domain, we also impose a uniform ellipticity condition on the operators $F_1$ and $F_2$. 

\begin{Assumption}[Uniform ellipticity]\label{assump_ellipticity}
The operators $F_i:S(d)\to\mathbb{R}$ are $(\lambda,\Lambda)$-elliptic, for $i=1,2$. That is, for every $M,P\in S(d)$, with $P\geq 0$, we have
\[
	F_i(M)-\Lambda\tr(P)\leq F_i(M+P)\leq F_i(M)-\lambda\tr(P),
\]
for $i=1,2$. 
\end{Assumption}

It is useful to write (A\ref{assump_ellipticity}) in terms of the extremal operators $\mathcal{P}_{\lambda,\Lambda}^\pm$, defined as
\[
	\mathcal{P}_{\lambda,\Lambda}^+(M):=-\lambda\tr(M^+)+\Lambda\tr(M^-)
\]
and
\[
	\mathcal{P}_{\lambda,\Lambda}^-(M):=-\Lambda\tr(M^+)+\lambda\tr(M^-).
\]
Since the ellipticity constants are fixed throughout the paper, we drop the subscripts and write $\mathcal{P}^\pm_{\lambda,\Lambda}=\mathcal{P}^\pm$. The condition in (A\ref{assump_ellipticity}) then becomes
\[
	\mathcal{P}^-(M-N)\leq F_i(M)-F_i(N)\leq\mathcal{P}^+(M-N),
\]
for every $M,N\in S(d)$ and $i=1,2$.

To prove the existence of $L^p$-viscosity solutions to the Dirichlet problem \eqref{eq_main}, we only require Assumptions (A\ref{assump_domain}) and (A\ref{assump_ellipticity}). However, further regularity of these solutions depends on additional conditions on the operators $F_1$ and $F_2$. We proceed with an assumption on the proximity of those operators.

\begin{Assumption}[Closeness of operators]\label{assump_comp}
There exist constants $K,\tau>0$ such that
\begin{equation}\label{assump_localbehavior}
	\left|F_1(M)-F_2(M)\right|\leq K+\tau\|M\|,\quad \forall M\in S(d).
\end{equation}
\end{Assumption}

The constant $\tau$ in (A\ref{assump_comp}), as well as the constant $\sigma$ in Assumption (A\ref{assump_convexity}), will be required to be sufficiently small later in the paper; see Sections \ref{sec_proofthm1} and \ref{sec_proofthm2}.
We emphasize that existence of $L^p$-viscosity solutions to \eqref{eq_main} does not require the proximity regime in (A\ref{assump_comp}); we only resort to this condition to prove an estimate in H\"older spaces. The inequality in (A\ref{assump_comp}) is satisfied, for example, by linear operators in the non-divergence form governed by matrices that are close in some suitable topology. A fully nonlinear example can be found in the context of Isaacs equations. 

\begin{example}[Isaacs equations]\label{eq_exbell} 
Let $A_{\alpha,\beta},B_{\alpha,\beta}\in S(d)$ for all $\alpha\in\mathcal{A}, \beta\in \mathcal{B}$, where $\mathcal{A},\mathcal{B}$ are some sets. Suppose that there exist constants $0<\lambda\leq \Lambda$ such that
\[
	\lambda I\leq A_{\alpha,\beta},B_{\alpha,\beta}\leq \Lambda I,
\]
for all $(\alpha,\beta)\in (\mathcal{A},\mathcal{B})$, where $I$ is the identity matrix. Suppose further that there exists $\tau>0$ such that
\[
	\left|A_{\alpha,\beta}-B_{\alpha,\beta}\right|\leq \tau.
\]
Then the Isaacs operators 
\[
	\inf_{\alpha\in\mathcal{A}}\sup_{\beta\in\mathcal{B}}\left(-\tr\left(A_{\alpha,\beta}M\right)\right)\hspace{.2in}\mbox{and}\hspace{.2in}\inf_{\alpha\in\mathcal{A}}\sup_{\beta\in\mathcal{B}}\left(-\tr\left(B_\beta M\right)\right)
\]
satisfy \emph{(A\ref{assump_comp})} with $K=0$.
\end{example}

To prove the existence of strong solutions we impose an additional, natural, condition on the operators $F_1$ and $F_2$. We require them to be locally close to a convex operator.

\begin{Assumption}[Near convexity condition]\label{assump_convexity}
There exist a convex, $(\lambda,\Lambda)$-elliptic operator $F=F(M)$ and constants $L,\sigma>0$ such that
\begin{equation}\label{assump_localbehavior}
	\left|F_i(M)-F(M)\right|\leq L+\sigma\|M\|,\quad i=1,2, \forall M\in S(d).
\end{equation}
\end{Assumption}
To illustrate the requirement in Assumption (A\ref{assump_convexity}) we again discuss the example of Isaacs operators. 

\begin{example}\label{eq_isaacs}
Let $A_{\alpha,\beta}$ be as in Example \ref{eq_exbell}. Suppose there exist $\sigma>0$ and
matrices $A_\beta\in S(d)$ such that $\lambda I\leq A_\beta\leq \Lambda I$ and
\[
	\sup_{x\in\Omega}\left|A_{\alpha,\beta}-A_\beta\right|\leq \sigma,\quad\forall\alpha\in\mathcal{A},\beta\in\mathcal{B}.
\]
Then
\[
	\left|\inf_{\alpha\in\mathcal{A}}\sup_{\beta\in\mathcal{B}}\left(-\tr\left(A_{\alpha,\beta}M\right)\right)-\sup_{\beta\in\mathcal{B}}\left(-\tr\left(A_{\beta}M\right)\right)\right|\leq \sigma\|M\|.
\]
Thus the Isaacs operators satisfy \emph{(A\ref{assump_convexity})} with $L=0$, when compared to the Bellman operator $F(M)=\sup_{\beta\in\mathcal{B}}\left(-\tr\left(A_{\beta}M\right)\right)$, provided the associated matrices are close enough.
\end{example}

We continue by introducing auxiliary operators. Consider $v\in C(\overline \Omega)$ such that $v=g$ on $\partial\Omega$ and fix arbitrary $\varepsilon>0$. Let $h_\varepsilon^v\in C(\overline\Omega)$ be defined by
\[
	h_\varepsilon^v=g_\varepsilon^v*\eta_\varepsilon,
\]
where 
\[
g_\varepsilon^v=\max\left(\min\left(\frac{v+\varepsilon}{2\varepsilon},1\right),0\right) \quad\mbox{on}\,\,\Omega,
\]
$g_\varepsilon^v=0$ on $\mathbb{R}^d\setminus\Omega$, and $\eta_\varepsilon$ is the standard mollifier function.
Define the operator $G_\varepsilon^v:\Omega\times S(d)\to\mathbb{R}$ by
\begin{equation}\label{eq_eq1}
	G_\varepsilon^v(x,M):=h_\varepsilon^v(x)F_1(M)+(1-h_\varepsilon^v(x))F_2(M).
\end{equation}

The next lemma establishes important properties of the operator $G_\varepsilon^v$ and closes this section.

\begin{Lemma}\label{lem_gepv}
Let Assumption \emph{(A\ref{assump_ellipticity})} hold and $G_\varepsilon^v$ be defined by \eqref{eq_eq1}. Then $G_\varepsilon^v$ is a $(\lambda,\Lambda)$-elliptic operator and there exists a constant $K_\varepsilon^v>0$ such that
\begin{equation}\label{eq:a1}
	\left|G_\varepsilon^v(x,M)-G_\varepsilon^v(y,M)\right|\leq K_\varepsilon^v\left|x-y\right|(1+\|M\|),\quad\forall x,y\in\Omega,M\in S(d).
\end{equation}
If in addition \emph{(A\ref{assump_comp})} holds then
\begin{equation}\label{eq:a21}
	\left|G_\varepsilon^v(x,M)-G_\varepsilon^v(y,M)\right|\leq 2(K+\tau\|M\|),\quad\forall x,y\in\Omega,M\in S(d)
\end{equation}
and if \emph{(A\ref{assump_convexity})} holds, then 
\begin{equation}\label{eq:a2}
	\left|G_\varepsilon^v(x,M)-F(M)\right|\leq L+\sigma\|M\|,\quad\forall x\in\Omega,M\in S(d).
\end{equation}
\end{Lemma}
\begin{proof}
We start by verifying the first assertion in the lemma. Note that
\[
	\begin{split}
		G_\varepsilon^v(x,M+P)&=h_\varepsilon^v(x)F_1(M+P)+(1-h_\varepsilon^v(x))F_2(M+P)\\
			&\leq h_\varepsilon^v(x)\left(F_1(M)-\lambda\tr(P)\right)+(1-h_\varepsilon^v(x))\left(F_2(M)-\lambda\tr(P)\right)\\
			&= h_\varepsilon^v(x)F_1(M)+(1-h_\varepsilon^v(x))F_2(M)-\lambda\tr(P)\\
			&=G_\varepsilon^v(x,M)-\lambda\tr(P).
	\end{split}
\]
The remaining inequality follows from an entirely analogous argument. For the second claim, we observe that
\begin{equation}\label{eq_comp1}
	\begin{split}
		\left|G_\varepsilon^v(x,M)-G_\varepsilon^v(y,M)\right|
			&\leq |h_\varepsilon^v(x)-h_\varepsilon^v(y)|(|F_1(M)|+|F_2(M)|)\\
			&\leq K_\varepsilon^v\left|x-y\right|(1+\|M\|).
	\end{split}
\end{equation}
If (A\ref{assump_convexity}) holds, since for every $x\in\Omega$ and  $M\in S(d)$, $F(M)=h_\varepsilon^v(x)F(M)+(1-h_\varepsilon^v(x))F(M)$, we have
\[
\begin{split}
\left|G_\varepsilon^v(x,M)-F(M)\right|&\leq h_\varepsilon^v(x)|F_1(M)-F(M)|+(1-h_\varepsilon^v(x))|F_2(M)-F(M)|
\\
&
\leq L+\sigma\|M\|.
\end{split}
\]
Finally, if (A\ref{assump_comp}) holds, then
\[
\begin{split}
\left|G_\varepsilon^v(x,M)-G_\varepsilon^v(y,M)\right|&\leq \left|G_\varepsilon^v(x,M)-F_1(M)\right|+\left|F_1(M)-G_\varepsilon^v(y,M)\right|
\\
&
\leq 2(K+\tau\|M\|).
\end{split}
\]
\end{proof}

\section{Existence of $L^p$-viscosity solutions}\label{sec_proofthm1}

In this section we present the proof of Theorem \ref{thm_existence1}. We consider, for $\varepsilon>0$, the regularized problems 
\begin{equation}\label{eq_dir1eps}
	\begin{cases}
		G_\varepsilon^u(x,D^2u)=f&\hspace{.4in}\mbox{in}\hspace{.1in}\Omega\\
		u=g&\hspace{.4in}\mbox{on}\hspace{.1in}\partial\Omega,
	\end{cases}
\end{equation}
were $G_\varepsilon^u$ is defined as in \eqref{eq_eq1}. 

To obtain $L^p$-viscosity solutions to \eqref{eq_dir1eps}, we consider auxiliary problems
\begin{equation}\label{eq_dir1}
	\begin{cases}
		G_\varepsilon^v(x,D^2u)=f&\hspace{.4in}\mbox{in}\hspace{.1in}\Omega\\
		u=g&\hspace{.4in}\mbox{on}\hspace{.1in}\partial\Omega,
	\end{cases}
\end{equation}
for $v\in C(\overline \Omega)$. We first prove the existence of global sub and supersolutions $\underline u$ and $\overline u$ to \eqref{eq_dir1}, independent of $v\in C(\overline \Omega)$ such that $v=g$ on $\partial\Omega$ and $\varepsilon>0$. Because $f$ is not necessarily continuous, we do not construct explicit barriers. Our strategy relies on the ellipticity of $G_\varepsilon^v$, combined with the existence of strong solutions to the Dirichlet problems governed by the extremal operators. 
We then show the existence of a unique $L^p$-viscosity solution $\underline u\leq u_\varepsilon^v\leq \overline u$. Moreover, the solutions $u_\varepsilon^v$ are bounded in some $C^\alpha_{\rm{loc}}$-space, uniformly in $v\in C(\overline \Omega)$ such that $v=g$ on $\partial\Omega$ and $\varepsilon>0$. 

To complete the proof, we introduce two objects. First, consider the set $B\subset C(\overline{\Omega})$, given by
\begin{equation}\label{eq_setB}
	B:=\left\lbrace v\in C(\overline{\Omega})\;|\;\underline{u}\leq v\leq\overline{u} \right\rbrace.
\end{equation}
Then we define the operator $T$ on $B$ as follows. Given $v\in B$, we consider the (unique) $L^p$-viscosity solution  $u^v_\varepsilon$ to \eqref{eq_dir1} and set 
\begin{equation}\label{eq_defmapT}
	Tv:=u^v_\varepsilon. 
\end{equation}	
For every $\varepsilon>0$, we prove the existence of a fixed point for the map $T$. This gives an $L^p$-viscosity solution to \eqref{eq_dir1eps}. By letting $\varepsilon\to 0$, we then obtain an $L^p$-viscosity solution to \eqref{eq_main}. We continue with the existence of sub and supersolutions to \eqref{eq_dir1}.




\begin{Lemma}[Existence of sub and supersolutions]\label{lem:barriers}
Suppose Assumptions \emph{(A\ref{assump_domain})}, \emph{(A\ref{assump_ellipticity})} hold, $g\in C(\partial\Omega)$ and $f\in L^p(\Omega)$, for some $p>p_0$. Let $v\in C(\overline \Omega)$ be such that $v=g$ on $\partial\Omega$ and let $\varepsilon>0$. Let $G_\varepsilon^v$ be defined as in \eqref{eq_eq1}. Then there exist a strong \emph{(}and $L^p$-viscosity\emph{)} subsolution 
$\underline u\in W^{2,p}_{\rm{loc}}(\Omega)\cap C(\overline\Omega)$ of \eqref{eq_dir1} and
 a strong \emph{(}and $L^p$-viscosity\emph{)} supersolution $\overline u\in W^{2,p}_{\rm{loc}}(\Omega)\cap C(\overline\Omega)$ of \eqref{eq_dir1} such that $\underline u=
 \overline u=g$ on $\partial\Omega$. The functions $\underline u$ and $\overline u$ are independent of $v$ and $\varepsilon$.
\end{Lemma}
\begin{proof}
Functions $\underline u$ and $\overline u$ are easily constructed; see \cite{CKLS1999}. We take $\overline u\in W^{2,p}_{\rm{loc}}(\Omega)\cap C(\overline\Omega)$ to be the unique
strong solution to
\begin{equation}\label{eq:b1}
	\begin{cases}
		{\mathcal P}^-(D^2\overline u)=f(x)+|F_1(0)|+|F_2(0)|&\hspace{.4in}\mbox{in}\hspace{.1in}\Omega\\
		\overline u=g&\hspace{.4in}\mbox{on}\hspace{.1in}\partial\Omega
	\end{cases}
\end{equation}
(see \cite[Corollary 3.10]{CCKS1996}). It then follows from (A\ref{assump_ellipticity}) that for a.e. $x\in\Omega$
\[
\begin{split}
G_\varepsilon^v(x,D^2\overline u(x))&=G_\varepsilon^v(x,D^2\overline u(x))-G_\varepsilon^v(x,0)+G_\varepsilon^v(x,0)
\\
&
\geq
{\mathcal P}^-(D^2\overline u(x))+G_\varepsilon^v(x,0)\geq f(x).
\end{split}
\]
Similarly, we check that if $\underline u\in W^{2,p}_{\rm{loc}}(\Omega)\cap C(\overline\Omega)$ is the unique
strong solution to
\begin{equation}\label{eq:b1}
	\begin{cases}
		{\mathcal P}^+(D^2\underline u)=f(x)-|F_1(0)|-|F_2(0)|&\hspace{.4in}\mbox{in}\hspace{.1in}\Omega\\
		\overline u=g&\hspace{.4in}\mbox{on}\hspace{.1in}\partial\Omega,
	\end{cases}
\end{equation}
then $G_\varepsilon^v(x,D^2\underline u(x))\leq f(x)$ for a.e. $x\in\Omega$. It is then straightforward to verify that $\underline{u}$ and $\overline{u}$ are also $L^p$-viscosity sub and supersolutions \cite[Lemma 2.5]{CCKS1996} and to complete the proof.
\end{proof}

\begin{Proposition}\label{prop_existenceepv}
Let Assumptions \emph{(A\ref{assump_domain})}, \emph{(A\ref{assump_ellipticity})} hold, $f\in L^p(\Omega)$, $p>p_0$ and $g\in C(\partial\Omega)$. Let $v\in C(\overline \Omega)$ be such that $v=g$ on $\partial\Omega$ and let $\varepsilon>0$. Let $G_\varepsilon^v$ be defined as in \eqref{eq_eq1}. Then there exists a unique $L^p$-viscosity solution $u_\varepsilon^v\in C(\overline\Omega)$ to \eqref{eq_dir1}. The solution $u_\varepsilon^v$ satisfies
\begin{equation}\label{eq:b5}
\underline u\leq u_\varepsilon^v\leq \overline u.
\end{equation}
Finally, for every $\Omega'\Subset \Omega$, 
\begin{equation}\label{eq:b6}
	\left\|u_\varepsilon^v\right\|_{C^{\alpha}(\Omega')}\leq C\left(\left\|g\right\|_{L^\infty(\Omega)}+\left\|f\right\|_{L^p(\Omega)}+|F_1(0)|+|F_2(0)|\right).
\end{equation}
for some universal $\alpha>0$ and $C=C(d,\lambda,\Lambda,p,{\rm diam}(\Omega),{\rm dist}(\Omega',\partial\Omega))$, which does not depend on 
$\varepsilon$ and $v$.
\end{Proposition}
\begin{proof}
We first argue that comparison principle holds for \eqref{eq_dir1}. Let $u\in C(\overline\Omega)$ be an $L^p$-viscosity subsolution of \eqref{eq_dir1} and 
$v\in C(\overline\Omega)$ be an $L^p$-viscosity supersolution of \eqref{eq_dir1}. Let $f_n\in C(\overline\Omega)$ be functions such that 
$\|f-f_n\|_{L^p(\Omega)}\to 0$ as $n\to\infty$. Let $\psi_n\in W^{2,p}_{\rm{loc}}(\Omega)\cap C(\overline\Omega)$ be the strong solution to
\[
	\begin{cases}
		{\mathcal P}^-(D^2\psi_n)=f(x)-f_n(x)&\hspace{.4in}\mbox{in}\hspace{.1in}\Omega\\
		\psi=0&\hspace{.4in}\mbox{on}\hspace{.1in}\partial\Omega.
	\end{cases}
\]
We have 
\begin{equation}\label{eq:b2}
\|\psi_n\|_{L^\infty(\Omega)}\leq C\|f-f_n\|_{L^p(\Omega)}\to 0\quad\mbox{as}\,\,n\to\infty.
\end{equation}
Also, it is easy to see that the functions $u_n=u-\psi_n$ are $L^p$-viscosity subsolutions of
\begin{equation}\label{eq:b3}
	\begin{cases}
		G_\varepsilon^v(x,D^2w)=f_n(x)&\hspace{.4in}\mbox{in}\hspace{.1in}\Omega\\
		w=g&\hspace{.4in}\mbox{on}\hspace{.1in}\partial\Omega.
	\end{cases}
\end{equation}
Similarly, let $\varphi_n\in W^{2,p}_{\rm{loc}}(\Omega)\cap C(\overline\Omega)$ be the strong solution to
\[
	\begin{cases}
		{\mathcal P}^+(D^2\varphi_n)=f(x)-f_n(x)&\hspace{.4in}\mbox{in}\hspace{.1in}\Omega\\
		\varphi=0&\hspace{.4in}\mbox{on}\hspace{.1in}\partial\Omega.
	\end{cases}
\]
Then
\begin{equation}\label{eq:b4}
\|\varphi_n\|_{L^\infty(\Omega)}\leq C\|f-f_n\|_{L^p(\Omega)}\to 0\quad\mbox{as}\,\,n\to\infty,
\end{equation}
and the functions $v_n=u-\varphi_n$ are $L^p$-viscosity supersolutions of \eqref{eq:b3}.
It is well known -- see for instance \cite[Theorem III.1-(1)]{Ishii-Lions1990}, together with Section V.1 there -- that \eqref{eq:b3} satisfies comparison principle; so we have
$u_n\leq v_n$ on $\overline\Omega$. Together with \eqref{eq:b2} and \eqref{eq:b4}, this fact implies $u\leq v$ on $\overline\Omega$. 
Thus there exists a
unique $L^p$-viscosity solution $u_\varepsilon^v\in C(\overline\Omega)$ to \eqref{eq_dir1}. Obviously \eqref{eq:b5} is satisfied
and \eqref{eq:b6} follows from standard interior regularity results for $L^p$-viscosity solutions.

The existence of an $L^p$-viscosity solution to \eqref{eq_dir1} is standard; see \cite{CKLS1999, JS, S1997}.
It is obtained by first producing standard viscosity solutions
$u_{\varepsilon,n}^v\in C(\overline\Omega)$ to \eqref{eq:b3} by Perron's method. 

Since, by the Aleksandrov-Bakelman-Pucci maximum principle for $L^p$-viscosity solutions (e.g. \cite[Proposition 3.3]{CCKS1996} or \cite[Lemma 1.4]{S1997}), 
\[
\|u_{\varepsilon,n}^v-u_{\varepsilon,m}^v\|_{L^\infty(\Omega)}\leq C\|f_n-f_m\|_{L^p(\Omega)}\to 0\quad\mbox{as}\,\,n,m\to\infty,
\]
the sequence $(u_{\varepsilon,n}^v)_{n=1}^\infty$ converges uniformly on $\overline\Omega$ to a function $u_\varepsilon^v\in C(\overline\Omega)$, satisfying $u_\varepsilon^v=g$ on $\partial\Omega$. Using the stability property of $L^p$-viscosity solutions, as in \cite[Theorem 3.8]{CCKS1996}, we then have that $u_\varepsilon^v$ is an $L^p$-viscosity solution to \eqref{eq_dir1}.
\end{proof}

In the sequel we examine the map $T:B\to C(\overline{\Omega})$.

\begin{Proposition}\label{prop_mapT}
Let $B\subset C(\overline \Omega)$ be defined as in \eqref{eq_setB} and $T:B\to C(\overline\Omega)$ be defined as in \eqref{eq_defmapT}. Then:
\begin{enumerate}
\item[{\rm (a)}] $B$ is a closed and convex subset of $C(\overline\Omega)$.
\item[{\rm (b)}] $T(B)\subset B$ and $T(B)$ is a precompact subset in $C(\overline\Omega)$.
\item[{\rm (c)}] The map $T:B\to B$ is continuous. 
\end{enumerate}
\end{Proposition}
\begin{proof}
We notice that it is obvious from the definition of $B$ (see \eqref{eq_setB}) that $B$ is a closed and convex. 

To establish {\rm{(b)}}, we first observe that Proposition \ref{prop_existenceepv} yields $\underline u\leq Tv\leq \overline u$ so $T(B)\subset B$. 
To show that $T(B)$ is precompact in $C(\overline\Omega)$, we proceed as follows. Take a sequence $(Tv_n)_{n\in\mathbb{N}}\subset T(B)$. Proposition \ref{prop_existenceepv} ensures that $(Tv_n)_{n\in\mathbb{N}}$ is equicontinuous in $C(\overline\Omega)$. Hence, it admits a convergent subsequence in $B$, which completes the argument.

To complete the proof, we verify {\rm{(c)}}. Take $(v_n)_{n\in\mathbb{N}}\subset B$ and suppose $v_n\to v\in B$ in $C(\overline\Omega)$. We need to prove that $Tv_n\to Tv$ in $C(\overline\Omega)$. 

First, we claim that the sequence of operators $(G_\varepsilon^{v_n})_{n\in\mathbb{N}}$ converges locally uniformly to $G_\varepsilon^v$. In fact,
\[
	\begin{split}
	\sup_{\Omega}\left|h_\varepsilon^{v_n}(x)-h_\varepsilon^v(x)\right|
	\leq\frac{1}{2\varepsilon}\sup_{\Omega}\left|v_n(x)-v(x)\right|\to 0\quad\mbox{as}\,\,n\to\infty.
	\end{split}
\]
Hence, $h_\varepsilon^{v_n}\to h_\varepsilon^{v}$ uniformly in $\Omega$. Therefore, the definition of the operators ensures that $G_\varepsilon^{v_n}\to G_\varepsilon^v$ locally uniformly. 

Since $T(B)$ is precompact in $C(\overline\Omega)$, there exists $w\in B$ such that $Tv_n\to w$, through a subsequence 
$(Tv_{n_i})_{i\in\mathbb{N}}$ if necessary. The convergence of the sequence $(G_\varepsilon^{v_n})_{n\in\mathbb{N}}$, together with the stability of viscosity solutions, ensures that $w$ is an $L^p$-viscosity solution to
\[
	G_\varepsilon^v(x,D^2w)=f\hspace{.2in}\mbox{in}\hspace{.2in}\Omega,
\]
with $w=g$ on $\partial\Omega$. The uniqueness of $L^p$-viscosity solutions to \eqref{eq_dir1} yields $Tv=w$. 

Finally, we notice the previous argument does not depend on the subsequence $(Tv_{n_i})_{i\in\mathbb{N}}$. In fact, suppose there exists $z\in B$ such that, through a different subsequence $(Tv_{n_j})_{j\in\mathbb{N}}$, we have $Tv_{n_j}\to z$, as $j\to \infty$. As before, the uniqueness of solutions to \eqref{eq_dir1} ensures that $z=w=Tv$ and the proof is complete.
\end{proof}

We are now in a position to prove Theorem \ref{thm_existence1}.

\begin{proof}[Proof of Theorem \ref{thm_existence1}]
We split the proof into two steps. First we consider the existence of solutions to \eqref{eq_main}.

\medskip

\noindent{\bf Step 1 - }Proposition \ref{prop_mapT} ensures that we can use the Schauder Fixed Point Theorem (see for instance \cite[Corollary 11.2]{GT2001}) to obtain a fixed point $u_\varepsilon$ of $T$. Thus
$u_\varepsilon$ is an $L^p$-viscosity solution to \eqref{eq_dir1eps}.
We now choose a convergent subsequence $u_{\varepsilon_n}$ such that $u_{\varepsilon_n}\to u$ in $C(\overline\Omega)$ as $\varepsilon_n\to 0$ for some $u$ in $C(\overline\Omega)$. Since $G_{\varepsilon_n}^{u_{\varepsilon_n}}$ converges locally uniformly on $(\left(\Omega^+(u)\cup\Omega^-(u)\right)\cap \Omega)\times S(d)$ to $G^u$ as 
$\varepsilon_n\to 0$,
where $G^u(x,M)=F_1(M)\chi_{\{u(x)>0\}}+F_2(M)\chi_{\{u(x)<0\}}$, it is easy to see that $u$ is an $L^p$-viscosity solution to \eqref{eq_main}.

\medskip

\noindent{\bf Step 2 - }It remains to prove that, if (A\ref{assump_comp}) holds and $\tau$ is sufficiently small, $u\in C^{1,\alpha}_{\rm{loc}}(\Omega)$,
 for some $\alpha\in(0,1)$, and the estimate in \eqref{eq_c1aest} is available. We generally repeat the strategy of the proof of \cite[Theorem 8.3]{Caffarelli-Cabre1995}. We start with a few observations. We define
\[
	\hat{G}_\varepsilon^{u_{\varepsilon}}(x,M):= G_\varepsilon^{u_{\varepsilon}}(x,M)-G_\varepsilon^{u_{\varepsilon}}(x,0)\hspace{.2in}\mbox{and}\hspace{.2in}\hat{f}(x):=f(x)-G_\varepsilon^{u_{\varepsilon}}(x,0)
\]
and notice that $u_\varepsilon$ is the $L^p$-viscosity solution to 
\[
	\hat{G}_\varepsilon^{u_{\varepsilon}}(x,D^2u_\varepsilon)=\hat{f}(x)
\]
in $\Omega$, with $u_\varepsilon=g$ on $\partial\Omega$. We note that the Aleksandrov-Bakelman-Pucci maximum principle for $L^p$-viscosity solutions yields
\[
\|u_\varepsilon\|_{L^\infty(\Omega)}\leq C(d,p,\lambda,\Lambda,{\rm diam}(\Omega))(|F_1(0)|+|F_2(0)|+\|f\|_{L^p(\Omega)}+\|g\|_{L^\infty(\partial\Omega)}).
\]

For every $0<r<1$, the set $\Omega'$ can be covered by a finite number of open balls $B_{\frac{r}{2}}(x_i), i=1,...,m$, for some $x_i\in\Omega'$ and 
such that
$\overline B_{r}(x_i)\subset\Omega$. Thus it is enough to prove the result for $\Omega'=B_{\frac{r}{2}}(x_i)$ for one of such balls. 
To simplify notation we will assume that $x_i=0\in\Omega'$.

Now we introduce a scaling. Set 
\[
	\tilde{u}_\varepsilon(x):=\frac{1}{N}u_\varepsilon(rx),
\]
where $0<r\ll1$ will be determined later and 
\[
	N:=1+\|\hat f\|_{L^p(\Omega)}+\left\|u_\varepsilon\right\|_{L^\infty(\Omega)}.
\]
We observe that $\tilde{u}_\varepsilon^v$ is the unique $L^p$-viscosity solution to 
\[
		\tilde{G}_\varepsilon^{u_{\varepsilon}}(x,D^2\tilde{u}_\varepsilon)=\tilde{f}(x)\hspace{.2in}\mbox{in}\hspace{.2in}B_1,
\]
where
\begin{equation}\label{eq:Gtilde}
	\tilde{G}_\varepsilon^{u_{\varepsilon}}(x,M):=\frac{r^2}{N}\hat{G}_\varepsilon^{u_{\varepsilon}}\left(rx,\frac{N}{r^2}M\right)\hspace{.2in}\mbox{and}\hspace{.2in}\tilde{f}(x):=\frac{r^2}{N}\hat{f}(rx).
\end{equation}
We have $\left\|\tilde{u}_\varepsilon\right\|_{L^\infty(B_1)}\leq 1$ and $\|\tilde{f}\|_{L^p(B_1)}\leq 1$. Also, by choosing $0<r\ll1$, we can make $\|\tilde{f}\|_{L^p(B_1)}$ arbitrarily small. 
Finally, we define $\beta:\Omega\times\Omega\to\mathbb{R}$ as
\[
	\beta^\varepsilon(x,x_0):=\sup_{M\in S(d)}\frac{|\tilde{G}_\varepsilon^{u_{\varepsilon}}(x,M)-\tilde{G}_\varepsilon^{u_{\varepsilon}}(y,M)|}{1+\|M\|}.
\]
By \eqref{eq:a21} and the definition of $\tilde{G}_\varepsilon^{u_{\varepsilon}}$, we have 
\[
	\begin{split}
		\frac{|\tilde{G}_\varepsilon^{u_{\varepsilon}}(x,M)-\tilde{G}_\varepsilon^{u_{\varepsilon}}(y,M)|}{1+\|M\|}
			&\leq \frac{2\frac{r^2}{N}\left(2K+\tau\frac{N}{r^2}\|M\|\right)}{1+\|M\|}\\
			&\leq \frac{2\tau\left(\frac{2Kr^2}{N\tau}+\|M\|\right)}{1+\|M\|}\\
			&\leq 2\tau,
	\end{split}
\]
provided we chose $K$ and $r$ such that $2Kr^2\leq N\tau$. Thus
\[
	\sup_{x,x_0\in\Omega}\beta^\varepsilon(x,x_0)\leq 2\tau.
\]
At this point we can use \cite[Theorem 8.3]{Caffarelli-Cabre1995} and \cite[Theorem 2.1]{S1997} to claim that there exists $0<\beta_0=\beta_0(d,p,\lambda,\Lambda,\alpha)\ll 1$ such that, if $0<\tau\leq \beta_0$, then $u_\varepsilon\in 
C^{1,\alpha}_{\rm{loc}}(\Omega)$ for $\alpha$ as in the statement of Theorem \ref{thm_existence1} and estimate \eqref{eq_c1aest} holds for $u_\varepsilon$. Since estimate \eqref{eq_c1aest} for $u_\varepsilon$ does not depend on $\varepsilon$, we then conclude that $u\in C^{1,\alpha}_{\rm{loc}}(\Omega)$ and it satisfies \eqref{eq_c1aest}.
\end{proof}

\section{Existence of strong solutions}\label{sec_proofthm2}

In this section we work under (A\ref{assump_convexity}) and establish the existence of strong solutions to \eqref{eq_main}. The main observation allowing us to prove this result is that in this case, the operator $G_\varepsilon^v$ satisfies the closeness property \eqref{eq:a2}. The latter allows us to frame \eqref{eq_dir1} in the context of the $W^{2,p}$-regularity theory \cite{Caffarelli1989,Caffarelli-Cabre1995,Escauriaza1993}.

\begin{Proposition}[Sobolev regularity for $u_\varepsilon$]\label{prop_reg2}
Let Assumptions \emph{(A\ref{assump_domain})}, \emph{(A\ref{assump_ellipticity})}, \emph{(A\ref{assump_convexity})} hold, $g\in C(\partial\Omega)$ and $f\in L^p(\Omega)$ for some $p>p_0$. Let $u_\varepsilon$ be an $L^p$-viscosity solution to \eqref{eq_dir1eps}. Then there exists $\beta_0=\beta_0(n,p,\lambda,\Lambda)>0$ such that if $\sigma\leq \beta_0$, then $u_\varepsilon\in W^{2,p}_{\rm{loc}}(\Omega)$ and it is a strong solution to \eqref{eq_dir1eps}. Moreover, for every $\Omega'\Subset \Omega$ there exists $C>0$ satisfying
\[
	\left\|u_\varepsilon\right\|_{W^{2,p}(\Omega')}\leq C\left(1+|F_1(0)|+|F_2(0)|+\|f\|_{L^p(\Omega)}+\|g\|_{L^\infty(\partial\Omega)}\right),
\]
where $C=C(d,p,\lambda,\Lambda,L,\sigma,{\rm dist}(\Omega',\partial\Omega),{\rm diam}(\Omega))$.
\end{Proposition}
\begin{proof}
The beginning of the proof is the same as Step 3 of the proof of Theorem \ref{thm_existence1} up to \eqref{eq:Gtilde}.
It is now enough to prove that 
\begin{equation}\label{eq:c1}
\|\tilde u_\varepsilon\|_{W^{2,p}(B_{1/2})}\leq C.
\end{equation} 

We define $\hat F(M)=F(M)-F(0)$ and
\[
\tilde F(M)={\frac{r^2}{N}}\hat F\left({\frac{N}{r^2}}M\right).
\]
The operator $\tilde F$ is convex with the same interior $C^{1,1}$ estimates as $\hat F$. That is, if $w\in C^2(B_1)\cap C(\overline B_1)$ is a solution
to
\[
\tilde F(D^2w)=0\quad\mbox{in}\,\, B_1
\]
then
\[
\|w\|_{C^{1,1}(B_{1/2})}\leq k\|w\|_{L^\infty(\partial B_1)}
\]
for some absolute constant $k=k(d,\lambda,\Lambda)$, independent of ${\frac{N}{r^2}}$.

We now define the function $\beta:\Omega\to\mathbb{R}$ by
\[
	\beta^\varepsilon(x):=\sup_{M\in S(d)}\frac{\left|\tilde G_\varepsilon^{u_{\varepsilon}}(x,M)-\tilde F(M)\right|}{1+\left\|M\right\|}.
\]
Using \eqref{eq:a2} we easily compute
\[
\begin{split}
\frac{\left|\tilde G_\varepsilon^{u_{\varepsilon}}(x,M)-\tilde F(M)\right|}{1+\left\|M\right\|}
&\leq \frac{\frac{r^2}{N}(2L+\sigma \frac{N}{r^2}\|M\|)}{1+\left\|M\right\|}
\\
&
\leq \frac{\sigma(\frac{2Lr^2}{N\sigma}+\|M\|)}{1+\left\|M\right\|}\leq \sigma
\end{split}
\]
if $2Lr^2\leq N\sigma$. We can now repeat the arguments of the proof of \cite[Proposition 7.2]{Caffarelli-Cabre1995} (see also \cite{Caffarelli1989} and \cite{Escauriaza1993} for the case $p_0<p\leq n$) with some straightforward modifications to obtain that there exists $\beta_0=\beta_0(d,p,\lambda,\Lambda)>0$ such that if $\|\tilde f\|_{L^p(B_1)}\leq \beta_0$ and $\sigma\leq \beta_0$, then \eqref{eq:c1} holds (see also Remark \ref{rem:1}). The fact that $u_\varepsilon$
is a strong solution to \eqref{eq_dir1} is the consequence of \cite[Corollary 3.7]{CCKS1996}.
\end{proof}

\begin{proof}[Proof of Theorem 2] 
Since the interior $W^{2,p}$ estimates of Proposition \ref{prop_reg2} are independent of $\varepsilon$, they will be satisfied by the $L^p$-viscosity solution $u$ 
to \eqref{eq_main} obtained in Step 1 of the proof of Theorem \ref{thm_existence1}. The function $u$ is then a strong solution to \eqref{eq_main} by \cite[Corollary 3.7]{CCKS1996}.
\end{proof}

\begin{Remark}\label{rem:1}\normalfont
A careful examination of the proof of \cite[Theorem 7.1]{Caffarelli-Cabre1995} shows that interior $W^{2,p}$ estimates for $L^p$-viscosity solutions to
\begin{equation}\label{eq:generalF}
F(x, D^2u)=f(x)\quad\mbox{in}\,\,\Omega,
\end{equation}
where $F$ is $(\lambda,\Lambda)$-elliptic, $F(x,0)=0$ in $\Omega$ and $f\in L^p(\Omega), p>p_0$, holds under the following conditions. First, one assumes that for every $x_0\in\Omega$ there exists a $(\lambda,\Lambda)$-elliptic function $F_{x_0}$ such that $F_{x_0}(0)=0$ and the equation $F_{x_0}(D^2w)=0$ satisfies interior $C^{1,1}$ estimates; that is if $w_0\in C(\partial B_1)$ then equation
\[
	\begin{cases}
		F_{x_0}(D^2w)=0&\hspace{.4in}\mbox{in}\hspace{.1in}B_1\\
		w=w_0&\hspace{.4in}\mbox{on}\hspace{.1in}\partial B_1
	\end{cases}
\]
has a smooth solution $w\in C^2(B_1)\cap C(\overline B_1)$ such that
\[
\|w\|_{C^{1,1}(B_{1/2})}\leq k\|w_0\|_{L^\infty(\partial B_1)}
\]
for some absolute constant $k$. The second condition is about the oscillation functions
$\beta(\cdot,x_0):\Omega\to\mathbb{R}$,
\begin{equation}\label{eq:betadef}
	\beta(x,x_0):=\sup_{M\in S(d)}\frac{\left|F(x,M)-F_{x_0}(M)\right|}{1+\left\|M\right\|}.
\end{equation}
One takes $r_0>0$ and considers the quantity
\begin{equation}\label{eq:beta}
\beta:=\sup\left\{\left(\frac{1}{r^d}\int_{B_r(x_0)}\beta(x,x_0)^pdx\right)^{\frac{1}{p}}:0<r<r_0,x_0\in\Omega, B_r(x_0)\subset\Omega\right\}.
\end{equation}
One can then prove that there exists $\beta_0=\beta_0(d,p,\lambda,\Lambda,k)$ such that if $\beta\leq \beta_0$, then an $L^p$-viscosity solution $u$ to
$F(x,D^2u)=f(x)$ in $\Omega$ belongs to $W^{2,p}_{\rm{loc}}(\Omega)$ and for
every open subset $\Omega'\Subset \Omega$,
\[
\|u\|_{W^{2,p}(\Omega')}\leq C\left(1+\|f\|_{L^p(\Omega)}
+\|u\|_{L^\infty(\Omega)}\right),
\]
where $C=C(d,p,\lambda,\Lambda,r_0,k,{\rm diam}(\Omega),{\rm dist}(\Omega',\partial\Omega))$.

In \cite[Theorem 7.1]{Caffarelli-Cabre1995} and \cite{Caffarelli1989}, $F_{x_0}(M)$ is replaced by $F(x_0,M)$ in the definition of $\beta(x,x_0)$; also
there is $\|M\|$ in the denominator in these papers instead of $1+\left\|M\right\|$ but this only affects the final form of the estimate. To obtain the result here
one needs to notice that in the proof of the key approximation Lemma 7.9 in \cite{Caffarelli-Cabre1995}, equation $F_0(D^2h)=0$ can be used instead of
$F(0,D^2h)=0$ and rescaled versions of functions $F_{x_0}(M)$ can be used in the proof there instead of rescaled versions of $F(x_0,M)$. The changes needed 
to obtain the result in the case $p_0<p$ are explained in \cite{Escauriaza1993}.

A condition which is essentially the same as the smallness of \eqref{eq:beta} (after rescaling if necessary) has been used before to obtain $W^{2,p}$ regularity results in \cite{Krylov2020,Krylov2018,Krylov2017,Krylov2013}, even though it is formulated there in a slightly different form. Also a similar condition was used in \cite{Huang2019} to obtain $W^{2,p,\mu}$ and $W^{2,\text{BMO}}$ regularity results. A related condition using a recession function of $F$ was also used in \cite{Pimentel-Teixeira2016}; see also \cite{Silvestre-Teixeira2015}. Finally, a condition similar to the smallness of \eqref{eq:beta} was used in the context of the Isaacs equation in \cite{Pimentel2019}.

We also remark that to obtain interior $C^{1,\alpha}$ estimates for $L^p$-viscosity solutions to \eqref{eq:generalF} it is enough to assume the smallness of
$\beta$ defined in \eqref{eq:beta}, where $\beta(x,x_0)$ is defined by \eqref{eq:betadef} and where the functions $F_{x_0}$ are any $(\lambda,\Lambda)$-elliptic
operators. A condition similar to \eqref{eq:a21} was used to prove $C^{1,\alpha}$ estimates for $L^p$-viscosity solutions in \cite{Krylov2018,Krylov2018b}. 
\end{Remark}

\bigskip

\noindent{\bf Acknowledgments} This work was partially supported by the Centre for Mathematics of the University of Coimbra - UIDB/00324/2020, funded by the Portuguese Government through FCT/MCTES. EP is partly funded by FAPERJ (E-26/200.002/2018), CNPq-Brazil (433623/2018-7, 307500/2017-9) and Instituto Serrapilheira (1811-25904). This study was financed in part by the Coordena\c{c}\~ao de Aperfei\c{c}oamento de Pessoal de N\'ivel Superior - Brasil (CAPES) - Finance Code 001.

\bigskip

\bibliography{bibliography}
\bibliographystyle{plain}

\bigskip

\noindent\textsc{Edgard A. Pimentel}\\
University of Coimbra\\
CMUC, Department of Mathematics\\ 
3001-501 Coimbra, Portugal\\
and\\
Pontifical Catholic University of Rio de Janeiro -- PUC-Rio\\
22451-900, G\'avea, Rio de Janeiro-RJ, Brazil\\
\noindent\texttt{edgard.pimentel@mat.uc.pt}

\vspace{.15in}

\noindent\textsc{Andrzej \'{S}wi\k{e}ch}\\
School of Mathematics\\
Georgia Institute of Technology\\
Atlanta GA 30332 USA\\
\noindent\texttt{swiech@math.gatech.edu}

\end{document}